\documentclass{amsart}

\usepackage{amsfonts,amsmath,amsthm}
\usepackage{enumitem} 
\usepackage{booktabs}

\usepackage{bm}

\usepackage{mathtools}    
\DeclarePairedDelimiterX\setc[2]{\{}{\}}{\,#1 \;\delimsize\vert\; #2\,}

\usepackage[dvipsnames]{xcolor}

\newcommand{\R}{\mathbb{R}}

\newcommand{\B}{\mathbb{B}}

\renewcommand{\S}{\mathbb{S}}

\newcommand{\Sn}{\mathbb{S}^n}

\newcommand{\blk}{\overline{\lambda}_k}

\renewcommand{\l}{\lambda}
\renewcommand{\O}{\Omega}
\newcommand{\h}{\tilde{h}}
\newcommand{\FF}{\tilde{F}}
\newcommand{\m}{d\tilde{\mu}}
\newcommand{\RN}[1]{%
  \textup{\uppercase\expandafter{\romannumeral#1}}%
}

\newcommand{\li}{\overline{\lambda}_{2k+1}(M,g)} 
\newcommand{\la}{\overline{\lambda}_{2k}(M,g)} 

\newtheorem{defi}{Definition}[section] 
\newtheorem{thm}[defi]{Theorem}

\newtheorem{rem}[defi]{Remark}
\newtheorem{prop}[defi]{Proposition}
\newtheorem{lemma}[defi]{Lemma}

\title{Payne-Pólya-Weinberger inequalities on closed Riemannian manifolds}

\author{{Mehdi} {Eddaoudi}}\email{mehdi.eddaoudi.1@ulaval.ca}\address{{D\'epartement de math\'ematiques et de statistique}, {Pavillon Alexandre-Vachon}, {Universit\'e Laval}, {Qu\'ebec QC}, {G1V 0A6}, {Canada}}

\begin{document}

\begin{abstract}
Payne-Pólya-Weinberger inequalities are known to be exclusive to bounded Euclidean domains with Dirichlet boundary condition. In this paper, we discuss the corresponding inequalities on Riemannian manifolds of dimension $n \geq3$, and we prove explicit bounds in terms of geometric quantities such as scalar curvature, Yamabe constant,  isoperimetric constant and  conformal volume. 
\end{abstract}

\maketitle

\section{Introduction}

\subsection{Universal inequalities for Dirichlet eigenvalues}
Let $\Omega \subset \R^n$ be a regular bounded domain, and consider the classical Dirichlet eigenvalue problem
\begin{equation}
    \begin{cases}
\Delta u = -\lambda^D u & \text{in } \Omega, \\
u = 0 & \text{on } \partial \Omega,
\end{cases} \label{eq:Dirichlet problem}
\end{equation}
It is well known that the spectrum of \eqref{eq:Dirichlet problem}  is real and discrete, and consists of a discrete sequence of eigenvalues
\[
\l_1^D(\O) < \l_2^D(\Omega) \leq \cdots \nearrow +\infty,
\]
repeated according to its multiplicity.

 Traditionally, an universal inequality establishes a relationship between Dirichlet eigenvalues that is independent of the specific geometric properties of the domain. This topic first appeared in the 1950s with a famous result by Payne, Pólya, and Weinberger (PPW for short)~\cite{PPW}, who proved an upper bound for the ratio  $\lambda^D_{2} / \lambda^D_1$ on planar bounded domains,
 
 \[
\frac{\lambda^D_{2}(\Omega)}{\lambda^D_{1}(\Omega)} \leq 3.
\]
They then conjectured that this ratio should achieve its maximum if and only if when the domain is an $n$-ball $\B^n$,
$$ \frac{\lambda^D_{2}(\O)}{\lambda^D_1(\O)} \leq \frac{\lambda^D_2(\mathbb{B}^n)}{\lambda^D_1(\mathbb{B}^n)}.
$$
This long-standing conjecture was eventually proven 35 years later in 1991 by Ashbaugh and Benguria~\cite{Ash-Ben,Ash-Ben2} following a series of improvements \cite{brands1964bounds,Thomp,vries1967upper,Hile-Protter,chiti1983bound,chiti1981inequalities,marcellini1980bounds,Yang}; see also \cite{ashbaugh1994isoperimetric,ashbaugh1996bounds,benguria2007second,Ash1999survol} for more literature on the topic. More generally, the PPW conjecture states that for $k\geq2$, 
\begin{equation}\label{Thompson}
    \frac{\lambda^D_{k+1}(\Omega)}{\lambda^D_k(\Omega)} < \frac{\lambda^D_2(\mathbb{B}^n)}{\lambda^D_1(\mathbb{B}^n)},
\end{equation}
 with equality achieved in the limit by a sequence of domains degenerating into $k$ disjoint $n$-balls of equal volume. This phenomenon of bubbling often arises in spectral optimization problems of various eigenvalue functionals \cite{GNP,BucurHenrot,Kim,EddaoudiGirouard,GL,GP,Petrides2014,karpukhinStern2023,bucur2022sharpsphere,KNPP1,KNPP2}, and it remains not yet fully understood to this day.

In higher dimension, the bound given by PPW becomes naturally $1+ \frac{4}{n}$, and Thompson ~\cite{Thomp} extended it in 1969 to  the general ratio $\l_{k+1}^D/\l_k^D$ by
$$\frac{\lambda^D_{k+1}(\Omega)}{\lambda^D_k(\Omega)} \leq 1 + \frac{4}{n}.
$$
Since then, other types of universal inequalities have been discovered over time. For example, we cite the inequality by Hile and Protter \cite{Hile-Protter} in 1980,
\[
\sum_{i=1}^{k} \frac{\lambda^D_{i}(\Omega)}{\lambda^D_{k+1}(\Omega) - \lambda^D_{i}(\Omega)} \geq \frac{kn}{4}.
\]
The quadratic inequality by Yang \cite{Yang} in 1991,
\begin{equation} \label{Yang}
    \sum_{i=1}^{k} (\lambda^D_{k+1} - \lambda^D_i)^2 \leq \frac{4}{n} \sum_{i=1}^{k} (\lambda^D_{k+1} - \lambda^D_i) \lambda^D_i.
\end{equation}
We have the following implications
\[
\text{(Yang)} \implies \text{(Hile-Protter)} \implies \text{(Thompson-PPW)}.
\]
There exists a broad spectrum of conjectures regarding universal inequalities, such as
  $$ \frac{\lambda_{2k}^D(\Omega)}{\lambda_k^D(\Omega)} \leq \frac{\lambda^D_{2}(\mathbb{\B }^n)}{\lambda^D_{1}(\mathbb{\B }^n)}.$$
     $$ \frac{\lambda_{n+2}^D(\Omega)}{\lambda_1^D(\Omega)} \leq \frac{\lambda^D_{n+2}(\mathbb{\B }^n)}{\lambda^D_{1}(\mathbb{\B }^n)}.$$
    $$ \frac{\lambda_2^D(\Omega) + \cdots + \lambda_{n+1}^D(\Omega)}{\lambda_1^D(\Omega)} \leq \frac{\lambda_2^D(\mathbb{B}^n) + \cdots + \lambda_{n+1}^D(\mathbb{B}^n)}{\lambda_1^D(\mathbb{B}^n)} = \frac{n \lambda^D_{2}(\mathbb{B}^n)}{\lambda^D_{1}(\mathbb{B}^n)}.$$
For additional open problems, see the survey by Ashbaugh and Benguria \cite{Ash1999survol}, and for further references on universal inequalities, see \cite{BrC,Chen-Yang-Zhen,Ch-Ze,Yang-Che-Qin,Ash}.

\subsection{PPW inequalities on closed Riemmanian manifolds}

Let $(M,g)$ be a smooth closed Riemannian manifold of dimension $n \geq 3$, and let
$$\lambda_0(M,g) = 0 < \lambda_1(M,g) \leq \lambda_2(M,g) \leq \cdots \nearrow \infty$$
denote the eigenvalues of the Laplace-Beltrami operator $\Delta_g$. 

Let $(f_j)_{j\geq 0}$ be an orthonormal basis of $L^2(M,g)$ corresponding to the eigenvalues $\lambda_j(M,g)$. We use one of the standard variational characterizations of $\lambda_k(M,g)$
\begin{gather}\label{caravariation}
\lambda_k(M,g)=\min_{f \in A_k\setminus\{0\}}\frac{\int_M|\nabla f|^2\,dv_g}{\int_M f^2\,dv_g},
\end{gather}
where $A_k$ is the following subspace of the Sobolev space $H^1(M)$: 
$$A_k = \setc*{ f \in  H^1(M)} {\int_M ff_j dv_g=0 \text{ for }j=0,1,\cdots,k-1}.$$ 
Functions in $A_k$ are said to be \emph{admissible} and they are used as trial functions in~\eqref{caravariation} to provide upper bounds on $\lambda_k(M,g)$.

Unlike the Euclidean case with Dirichlet boundary conditions, the functional \[ g \mapsto \frac{\lambda_{k+1}(M,g)}{\lambda_k(M,g)} \] is unbounded above. This can be illustrated for example with a manifold that degenerates into \( k+1 \) disconnected components such as Cheeger dumbbells. Indeed in this situation, one have $\lambda_k \to 0$, while $\lambda_{k+1} \to c$, where $c$ is non zero constant. Moreover, such behaviour can even occur when restricting to a given conformal class (bubbling phenomenon); thus, it is not entirely clear what would be the corresponding PPW inequality in the Riemannian setting.

However, El Soufi, Harrell, and Ilias \cite{EHI} (EHI for short) proved in 2007 an inequality that can be seen as Yang's inequality \eqref{Yang} for closed manifolds. It particularly features a geometric term: the mean curvature.

Let $X: M \to \mathbb{R}^m$ be an isometric immersion, and let $H$ be the mean curvature vector of $X$; the trace of its second fundamental form. Then for each $k \geq 0$
\begin{equation}\label{Yanggeneral}
    \sum_{i=0}^{k} (\lambda_{k+1} - \lambda_i)^2 \leq \frac{4}{n} \sum_{i=0}^{k} (\lambda_{k+1} - \lambda_i) \left( \lambda_i + \frac{\| |H|^2 \|_\infty}{4} \right).
\end{equation}
Here by abuse of notation, $\lambda_i := \lambda_i(M,g)$.

Other variants of this inequality can be found in \cite{Harr,Ily-Lap}. In the same fashion that Yang's inequality \eqref{Yang} implies Thompson's inequality \eqref{Thompson}, EHI's inequality \eqref{Yanggeneral} implies \cite[Corollary 2.1]{EHI}
\begin{equation}\label{EHI PPW inequality}
    \lambda_{k+1}(M,g)- \left(1+\frac{4}{n}\right) \lambda_k(M,g) \leq \frac{\| |H|^2\|_\infty}{n}.
\end{equation}

This last inequality is the starting point of this paper, where we consider inequalities of the form  
\[
\lambda_{k+1} - A \lambda_k \leq B,
\]
where \( A \) and \( B \) are explicit geometric quantities independent of \( k \). Observe that this inequality immediately implies that \( A > 1 \) and \( B > 0 \). Indeed, if \( B \leq 0 \), this would imply \(\frac{\lambda_{k+1}}{\lambda_k} \leq A\), which cannot hold in general, as discussed earlier. Similarly, if \( A \leq 1 \), it would suggest that \(\lambda_{k+1} - \lambda_k \leq B\), which is trivially false in many examples such as the standard sphere. We obtain the following results.

The first one is expressed within the conformal class of the canonical metric \( g_0 \) on the sphere \( \mathbb{S}^n \).
\begin{thm}\label{Thm1}
Let $(\S^n,g)$ be the sphere of dimension \( n \geq 3 \) equipped with a metric $g$ conformal to its canonical metric \( g_0 \). Then for all \( k \geq 1 \)
$$
\lambda_{2k+1}(M,g) - \left( 1+\frac{4}{n-2}\right) \lambda_{2k}(M,g) \leq \frac{1}{n-1} \max_{M} S_g.
$$
\end{thm}
Theorem \ref{Thm1} can be extended to Riemannian manifolds with a strictly positive Yamabe constant \( Y(M, [g]) \) through the concept of the \( m \)-conformal volume \( V_c(m, M, [g]) \) of Li and Yau \cite{Li-Yau}. Let us briefly recall their definitions: given a conformal immersion $\phi:M\to\S^m\subset\R^{m+1}$, let
$$V_c(m,\phi)=\sup_{\tau\in\text{Aut}(\S^m)}\text{vol}(M,(\tau\circ\phi)^{\star}g_0).$$
The \emph{$m$-conformal volume} of a conformal class $C=[g]$ is defined to be
$$V_c(m,M,C)=\inf_{\phi:M\to\S^m}V_c(m,\phi),$$
where the infimum is taken over all conformal immersions. By a combination of Nash embedding theorem and a stereographic projection, this is is well-defined once $m$ is large enough.
This geometric quantity has deep connections to the theory of $\l_1$-maximal maps and minimal surfaces, as discussed in \cite{ElSoufiIlias1984, el1986immersions, ElsoufiIlias2000,Li-Yau}.
Meanwhile the well-known Yamabe constant is defined as
\begin{equation}\label{funcionalY}
    Y(M, C) = \inf_{h \in C} \frac{\int_M S_h \, dv_h}{\left( \int_M dv_h \right)^{\frac{n-2}{n}}},
\end{equation}
where \( S_h \) is the scalar curvature of \( h \).

Both $Y(M, C)$ and $V_c(m,M,C)$ are conformal invariants.

\begin{thm}\label{Thm1bis}
Let $(M,g)$ be a closed Riemannian manifold of dimension $n\geq 3$, and let $C=[g]$. Suppose that $Y(M,C)>0$, then for any integer \( m > 0 \) such that the \( m \)-conformal volume of $C$ is well-defined and for all \( k \geq 1 \), 
$$
\lambda_{2k+1}(M,g) - \left( 1+\frac{4n(n-1)V_c(m,M,C)^{2/n}}{(n-2)Y(M,C)}\right) \lambda_{2k}(M,g) \leq \frac{nV_c(m,M,C)^{2/n}\|S_g\|_\infty }{Y(M,C)}.
$$
\end{thm}

\begin{rem}
    An important observation in Theorem \ref{Thm1} and Theorem \ref{Thm1bis} is the emergence of an intrinsic geometric quantity: scalar curvature. Scalar curvature is often regarded as the intrinsic counterpart to mean curvature \cite{GRO}, and numerous results establish comparison theorems under constraints involving one of these two quantities. In particular, it is well known that is can be prescribed within a conformal class by the Kazdan-Warner constraint formula \cite{kazdan1974curvature, kazdan1985prescribing, bourguignon1987scalar,Heb2}. To this end, Theorem \ref{Thm1} can be applied for example to construct metrics that exhibit a controlled gap between consecutive eigenvalues.

More generally, the estimate provided by the scalar curvature can be stronger than the one given by the mean curvature. For example, when \( M \) is a hypersurface of \( \mathbb{R}^{n+1} \) with principal curvatures \( (\kappa_i)_i \), Gauss equation leads to
\begin{equation}
S_g = \sum_{i \neq j} \kappa_i \kappa_j = \left( \sum_{i=1}^{n} \kappa_i \right)^2 - \sum_{i=1}^{n} \kappa_i^2 := |H|^2 - \| \RN{2} \|^2,
\end{equation}
where \( \RN{2} \) is the second fundamental form of the hypersurface \( M \).

Applying Schwarz inequality yields
\begin{equation}
\frac{1}{n-1} S_g \leq \frac{1}{n} |H|^2,
\end{equation}
with equality holding only at umbilical points, i.e., points where all the principal curvatures are equal. See \cite{ros1988compact, do1992riemannian, petersen2006riemannian} for further details on these relations.
\end{rem}
\begin{rem}
Whenever both \( Y(M,C) \) and \( V_c(m,M,C) \) can be computed, explicit bounds can be derived from Theorem~\ref{Thm1bis}. For example, the \( m \)-conformal volume is computable when there is a minimal immersion into \( \mathbb{S}^m \) by the first eigenfunctions \cite{el1986immersions}, such situations are covered in details in \cite{ElSoufiIlias1984,EddaoudiGirouard}. Additionally, a positive Yamabe constant implies that the conformal class contains a metric with positive scalar curvature---a consequence of the Yamabe problem \cite{yamabe1960deformation,schoen1984conformal,trudinger1968remarks,aubin1976equations}---and in such cases, it can also be computed.

\end{rem}

Instead of fixing a conformal class, our second result requires a positive lower bound on the Ricci curvature.
\smallskip

Denote by $$\blk(M,g) := \lambda_k(M,g) \text{vol}(M,g)^{2/n}$$ the \( k \)-th eigenvalue normalized by the volume.

\begin{thm}\label{Thm2}
Let \( (M,g) \) be a smooth closed manifold of dimension \( n \geq 3 \) such that \( \mathrm{Ric} \geq (n-1)a^2, \quad \text{with } a > 0 \). Then, for any integer \( m > 0 \) such that the \( m \)-conformal volume of $C=[g]$ is well-defined and for all \( k \geq 1 \),
$$\li -\la \left(1+\frac{4V_c(m,M,C)^\frac{2}{n}}{(n-2)a^2\text{vol}(M,g)^\frac{2}{n}} \right) \leq n V_c(m,M,C)^\frac{2}{n} .$$
\end{thm}
\begin{rem}
In general, imposing a lower bound on the Ricci curvature is a strong condition. For example, well-known estimates by Buser \cite{buser1982note} and Gromov \cite[Appendix C]{gromov1999metric}, as presented in Hassannezhad, Kokarev, and Polterovich \cite[Theorem 1.1 and Theorem 1.3]{hassannezhad2016eigenvalue}, state that when
\[
\operatorname{Ric} \geq -\kappa(n - 1)g,
\]
where \(\kappa > 0\), there exist constants \(C_1\) and \(C_2\) depending only on the dimension \(n\) of \(M\), such that for all \(k \geq 1\),
\begin{gather}
\lambda_k(M, g) \leq \frac{(n - 1)^2 \kappa}{4} + C_1 \left(\frac{k}{\text{vol}(M,g)}\right)^{2/n}, \label{eq:volume_bound} \\
\lambda_k(M, g) \geq C_2^{1 + d\sqrt{\kappa}} d^{-2} k^{2/n}, \label{eq:lower_bound}
\end{gather}
where \(d = d(M)\) is the diameter of \(M\).

Therefore by combining these upper and lower bounds, one can derive a relation of the form
\[
\lambda_{k+1}(M, g) - A \lambda_k(M, g) \leq B, \label{eq:combined_bound}
\]
where
\[
A = \frac{C_1 d^2}{C_2^{1 + d\sqrt{\kappa}} \text{vol}(M, g)^{2/n}}, \quad \text{and} \quad B = \frac{(n-1)^2 \kappa}{4} + (2^{2/n} - 1) \frac{C_1}{\text{vol}(M, g)^{2/n}}.
\]
However, these bounds are not effective, as the explicit expression of \(C_1\) and \(C_2\) is known to be very large.
\end{rem}

Finally, our third result is the most general one we obtain, and it involves several interesting geometric quantities interrelated. Let \( C(M,g) \) be the isoperimetric constant
$$
C(M,g) = \min_{\text{vol}(\Omega) \leq \frac{\text{vol}(M,g)}{2}} \frac{\text{vol}(\partial \Omega)}{\text{vol}(\Omega)^{\frac{n-1}{n}}},
$$
where \( \Omega \) varies among non-empty domains of \( M \).\\
And let \( C^* \) be its analogue for the unit ball \( \mathbb{B}^n \) in \( \mathbb{R}^n \)
$$
C^* = \frac{w_{n-1}}{\text{vol}(\mathbb{B}^n)^\frac{n-1}{n}}.
$$

\begin{thm}\label{Thm3}
Let \( (M,g) \) be a closed manifold of dimension \( n \geq 3 \). Then, for any integer \( m > 0 \) such that the \( m \)-conformal volume of \( C=[g] \) is well-defined, for all \( k \geq 1 \), we have
$$
\li - \la \left( 1 + \frac{8 C^{*2} V_c(m,M,C)^\frac{2}{n}}{(n-2) C^2(M,g) w_n^\frac{2}{n}} \right) \leq 4 n V_c(m,M,C)^\frac{2}{n}.
$$
\end{thm}

The proof of these results features a new construction of a vector field whose vanishing points serve as admissible functions in the variational characterization of eigenvalues. We guarantee the existence of these zero points through topological arguments, such as the center of mass and the Hopf-Poincaré theorem. Then, in order to estimate the Rayleigh quotient of admissible functions, we make use of key results on critical Sobolev embeddings with optimal constants. However, this method unfortunately cannot fully address all the differences between eigenvalues due to the topological nature of the Hopf-Poincaré argument. We could potentially solve this issue by adding a fold to our construction as done classically in \cite{nadirashvili2002isoperimetric,GNP,EddaoudiGirouard,Kim, Petrides2014}, although this might come at the cost of weakening the bound.

\subsection*{Outline of the paper} In section \ref{sectionSobolev}, we give a short overview on the Sobolev embeddings with critical exponents. Section \ref{section sphere} is devoted to a topological construction that yield to the proof of Theorem \ref{Thm1}. And finally, in section \ref{sectionfinale} we refine further this construction to arbitrary Riemannian manifolds and we prove Theorem \ref{Thm1bis}, Theorem \ref{Thm2} and Theorem \ref{Thm3}. 

\section{Sobolev Inequalities}\label{sectionSobolev}
Sobolev embeddings have always played a fundamental role in the study of partial differential equations (PDEs), and particularly in geometric analysis. One of their most significant applications appears in the  proof of the Yamabe problem, which seeks to find a conformal metric with constant scalar curvature in a compact Riemannian manifold. A key step in solving this problem involves sharp Sobolev inequalities, where the best constant in the Sobolev embedding plays a crucial role. 
The results of this section can be found in the thorough presentation of Hebey \cite{Heb} and Druet \cite{Druet} in their French version, or in their lecture note \cite{druet2002ab}. More generally, since this is a well-known topic, we refer to Hebey's books \cite{hebey2013compactness,hebey2000nonlinear,hebey1996sobolev}.

Let $n \geq 3$. by the Gagliardo-Nirenberg theorem, for any $1 \leq p < n$, the Sobolev space $H_1^p(\mathbb{R}^n)$ embeds continuously into $L^{p^*}(\mathbb{R}^n)$, where $p^* = \frac{np}{n-p}$.  

Let $K(n,p)$ be the norm of this embedding, defined by  
\[ K(n,p) = \sup_{f \in H^p(\mathbb{R}^n)} \frac{\|f\|_{p^*}}{\|\nabla f\|_p}. \]  
The explicit expression for $K(n,p)$ had been computed independently by Aubin \cite{T.Au} and Talenti \cite{talenti1976best}, see also \cite{Heb,cordero2004mass}.
For the purpose of this paper, we focus on the particular case where $p = 2$, in which case $p^* = \frac{2n}{n-2}$, denoted as $2^*$. The associated constant is given by  
\[ K(n,2) = \sqrt{\frac{4}{n(n-2)w_n^{\frac{2}{n}}}}, \]  
where $w_n$ is the volume of the unit sphere in $\mathbb{R}^{n+1}$. 

In the setting of a closed Riemannian manifold $(M,g)$ of dimension $n \geq 3$, an alternative formulation of Sobolev embeddings consists of finding two positive geometric constants $A$ and $B$ such that for all $f \in H^1(M)$
\begin{equation}\label{ABconstant}
    \|f\|_{2^*}^2 \leq A\|\nabla f\|_2^2 + B\|f\|_2^2.
\end{equation}  
The central question of determining the optimal constants $A$ and $B$ has been part of the celebrated $AB$-program \cite{druet2002ab}. A result of Hebey and Vaugon \cite{hebey1992meilleures,hebey1995best} asserts that in this case one can choose $A$ to be the optimal constant $K(n,2)^2$, and $B$ to be a geometric constant 
\begin{equation}\label{K(n,2)Bconstant}
    \|f\|_{2^*}^2 \leq K(n,2)^2\|\nabla f\|_2^2 + B\|f\|_2^2.
\end{equation} 
In some special case, the optimal constant $B$ can be explicitly computed. For example, for the standard sphere $(\S^n,g_0)$, Aubin \cite{aubin1976equations} proved that
\begin{equation}\label{Aubinsphereg_0}
    \|f\|_{2^*}^2 \leq K(n,2)^2\|\nabla f\|_2^2 + w_n^{-2/n}\|f\|_2^2.
\end{equation} 
This is optimal and equality is attained by bubble functions. Hebey \cite{Heb} extended this inequality to the conformal class of $[g_0]$. For all $g\in [g_0]$,
\begin{equation}\label{Hebey-Aubinsphere}
\|f\|^2_{2^*} \leq K^2(n,2)\|\nabla f\|^2_2 + \frac{n-2}{4(n-1)}K^2(n,2)\max_{\Sn} S_g \|f\|^2_2,  \end{equation}
where $S_g$ denotes the scalar curvature of $g$.

For $n \geq4$, Hebey's result is optimal, but it remains strict for $n=3$ \cite{hebey1998fonctions}. More generally, such inequalities are directly linked to the Yamabe problem, which seeks to find a metric with constant scalar curvature within a given conformal class. Through the transformation of scalar curvature in a conformal class, Yamabe introduced the functional 
\begin{equation}\label{Ydefivaria}
    Y(M, [g]) = \inf_{f \in H^1(M) \setminus \{0\}} \frac{\int_M \left( 4 \frac{n-1}{n-2} |\nabla f|^2 + S_g f^2 \right) dv_g}{\left( \int_M |f|^{\frac{2n}{n-2}} dv_g \right)^{\frac{n-2}{n}}},
\end{equation}
which is equivalent to the definition given in \eqref{funcionalY}.
Knowing that critical metrics of this functional have constant scalar curvature, Yamabe \cite{yamabe1960deformation} attempted to prove that its minimum is attained using the compactness of Sobolev embeddings. However, his proof contained a mistake, as it did not account for the lack of compactness at the critical Sobolev exponent. Nevertheless, his approach remained valid in specific cases, and it would take several years later to address this issue by Trudinger \cite{trudinger1968remarks}, Aubin \cite{aubin1976equations}, and Schoen \cite{schoen1984conformal}.
When $Y(M,[g])>0$, for all $f \in H^1(M)$, inequality \eqref{Ydefivaria} can be written as 
\begin{equation}\label{SobolevY}
    \|f\|_{2^*}^2 \leq \frac{4(n-1)}{(n-2)Y(M, [g])} \|\nabla u\|_{2}^2 + \frac{1}{ Y(M, [g])} \int_M S_g f^2 \, dv_g.
\end{equation} 
Therefore, using the expression of the Yamabe constant, this yields many examples of Sobolev inequalities with explicit constants $A$ and $B$.

The two next results were proven by Ilias \cite{ILIAS} using symmetrizations via Lévy-Gromov’s inequality \cite{lévy1951problèmes}.

The first one is given when Ricci curvature is bounded below positively, $Ric_g \geq (n-1)a^2$, where $a >0$. For all $f \in H^1(M)$ we have, 
\begin{equation}\label{IliasRicc}
\|f\|^2_{2^*} \leq \frac{4}{n(n-2)a^2 \text{vol}(M,g)^{\frac{2}{n}}} \|\nabla f\|^2_2 + \text{vol}(M,g)^{-\frac{2}{n}} \|f\|^2_2.  
\end{equation}
And the second one is obtained in the most general setting without any assumption on curvature,
\begin{equation}\label{Iliasgene}
\|f\|^2_{2^*} \leq 2K(n,2)^2\frac{C{^*}^2}{C^2(M,g)}\|\nabla f\|^2_2 + 4 \text{vol}(M,g)^\frac{-2}{n}\|f\|^2_2,
\end{equation}
where $C(M,g)$ is the isoperimetric constant and $C^*$ is it analogue for the ball $\B^n$.

\section{the Sphere case}\label{section sphere}
\subsection{Admissible functions as zeros of a vector field}

In this section, we establish a construction for the sphere \(\mathbb{S}^n\) that allows us to use trial functions in the variational characterization of eigenvalues \eqref{caravariation}. 
\smallskip

Let \( n \geq 3 \), and consider (\(\mathbb{S}^n,g)\) equipped with a metric \(g\) conformal to the canonical metric \( g_0 \). Let \( \{f_i\} \) denote an orthonormal basis of eigenfunctions for \( L^2(\mathbb{S}^n, g) \).

For any integer \( k \geq 1 \), consider a point \( p \in \mathbb{S}^{2k} \) on a parametrized sphere with coordinates \( (p_i)_{i=0, \dots, 2k} \). We define a density measure \( d\mu_p \) on \( \mathbb{S}^n \) as
\begin{equation}\label{densitédmu}
d\mu_p := \left( \sum_{i=0}^{2k} p_i f_i \right)^2 dv_g.
\end{equation}
By a standard topological argument based on the center of mass, see Laugesen \cite{laugesen2021well} for a recent version, for every \(p \in \mathbb{S}^{2k}\), there exists a unique point \(\xi_p \in \mathbb{B}^{n+1}\) such that
\[
\int_{\mathbb{S}^n} \phi_{\xi_p} \, d\mu_p = 0,
\]
where for each \(\xi \in \mathbb{B}^{n+1} \subset \mathbb{R}^{n+1}\), the family of conformal automorphisms \(\phi_\xi : \mathbb{S}^n \to \mathbb{S}^n\) is defined as
\begin{equation}\label{defiphi_xi}
    \phi_{\xi}(x) = \xi + \frac{1 - |\xi|^2}{|x + \xi|^2}(x + \xi).
\end{equation}

Let \(\{b_i\}\) be an orthonormal basis of \(\mathbb{R}^{n+1}\), and let \(X_{b_i}\) denote the coordinate functions of \(\mathbb{R}^{n+1}\).

We aim to construct a map whose coordinates satisfy the orthogonality conditions in the variational characterization of \(\lambda_{2k+1}(\mathbb{S}^n, g)\) as in \eqref{caravariation}. To this end,  we define the map \(\mathcal{F}: \mathbb{S}^{2k} \to \mathbb{R}^{2k+1}\) as
\begin{equation}\label{applicationF}  
    \mathcal{F}(p) =
\begin{bmatrix}  
\int_{\mathbb{S}^n} h(p,x) f_0(x) \, dv_g(x) \\  
\vdots \\  
\int_{\mathbb{S}^n} h(p,x) f_{2k}(x) \, dv_g(x)  
\end{bmatrix},
\end{equation}
where for every \((p,x) \in \mathbb{S}^{2k} \times \mathbb{S}^n\) the function \(h: \mathbb{S}^{2k} \times \mathbb{S}^n \to \mathbb{R}\) is defined as
\[
h(p,x) := \left(\sum_{i=1}^{n+1} X_{b_i} \circ \phi_{\xi_p}(x)\right) \left(\sum_{i=0}^{2k} p_i f_i(x)\right).
\]
Our objective is to show that \(\mathcal{F}\) has at least one vanishing point. The following lemma establishes this.

\begin{lemma}\label{lemmedecontinuite}
   There exists a point \(q \in \mathbb{S}^{2k}\) such that $$\mathcal{F}(q)=0.$$
\end{lemma}

\begin{proof}[Proof of Lemma \ref{lemmedecontinuite}]
We first show that \(\mathcal{F}\) is a continuous vector field on \(\mathbb{S}^{2k}\). To achieve this, we need to check that  
\[
\langle \mathcal{F}(p), p \rangle_{\mathbb{R}^{2k+1}} = 0.
\]
By the linearity of the integral, we have  
\begin{align*}
      \langle \mathcal{F}(p), p \rangle_{\mathbb{R}^{2k+1}} &= 
     \sum_{i=0}^{2k} \left(\int_{\mathbb{S}^n} h(p,x) f_i(x) \, dv_g(x) \right) p_i \\
     &= \int_{\mathbb{S}^n} h(p,x) \sum_{i=0}^{2k} p_i f_i(x) \, dv_g(x) \\
     &= \int_{\mathbb{S}^n} \left(\sum_{i=1}^{n+1} X_{b_i} \circ \phi_{\xi_p}(x)\right) \left(\sum_{i=0}^{2k} p_i f_i(x)\right)^2 dv_g(x) \\
     &= \sum_{i=1}^{n+1} \int_{\mathbb{S}^n} X_{b_i} \circ \phi_{\xi_p}(x) \, d\mu_p(x) \\
     &= 0,
\end{align*}
where the last step follows from the definition of \(\xi_p\), as it is the renormalization point of the measure \(d\mu_p\).  
The continuity of \(\mathcal{F}\) follows naturally from the continuity of the parameter \(\xi_p\), see Laugesen \cite{laugesen2021well}.
Thus, \(\mathcal{F}\) is a continuous vector field.

The proof is concluded by the Hopf-Poincaré theorem \cite{Hopf1927}. Indeed since \(\mathcal{F}\) is a continuous vector field on an even-dimensional sphere, there exists a point \(q \in \mathbb{S}^{2k}\) such that \(\mathcal{F}(q) = 0\).

\end{proof}
From now on, let \(q\) be a zero of \(\mathcal{F}\), as guaranteed by Lemma \ref{lemmedecontinuite}. Consequently, the function \(h(q,x)\) is an admissible function in the variational characterization of \(\lambda_{2k+1}(\mathbb{S}^n, g)\).

Consider the symmetric bilinear form \(\mathcal{G}_q\) on \(\mathbb{R}^{n+1}\), defined as
{\scriptsize
\begin{align}\label{formesbilinéaires}
\mathcal{G}_q(v,w) &= \lambda_{2k+1}(\S^n,g) \int_{\Sn} X_v \circ \phi_{\xi_q} X_w \circ \phi_{\xi_q} d\mu_q \\
& \quad - \int_{\Sn} \nabla \left(X_v \circ \phi_{\xi_q} \sum_{i=0}^{2k} q_i f_i\right)  \nabla \left(X_w \circ \phi_{\xi_q} \sum_{i=0}^{2k} q_i f_i \right) dv_g. \notag
\end{align}
}
By the diagonalization theorem, there exists an orthonormal basis \(\{e_i\}\) such that the bilinear form \(\mathcal{G}_q\) is diagonal in this basis.

Without loss of generality, we may assume that the function 
\[
f(x) := \sum_{i=1}^{n+1} X_{e_i} \circ \phi_{\xi_q}(x)  \sum_{i=0}^{2k} q_i f_i(x) 
\]
is also an admissible function for \(\lambda_{2k+1}(\mathbb{S}^n, g)\). This follows from the fact that the bases \(\{e_i\}\) and \(\{b_i\}\) are related by an orthogonal transformation, which preserves integrals over \(\mathbb{S}^n\), and that \(h\) is already an admissible function.

\subsection{Estimates of the Rayleigh Quotient of \(f\)}
By the variational characterization of \(\lambda_{2k+1}(\S^n, g)\) \eqref{caravariation}, the function $f$ verifies 

\begin{equation}\label{inégalitél2k+1}
\lambda_{2k+1}(\Sn, g) \int_{\Sn} f^2 \, dv_g \leq \int_{\Sn} |\nabla f|^2 \, dv_g.
\end{equation}
We first start by simplifying this expression.
Since
\begin{align*}
    \int_{\Sn} \left(\sum_{i=1}^{n+1}X^2_{e_i}\circ \phi_{\xi_q}\right)^2 \left(\sum_{i=0}^{2k} q_i f_i\right)^2 dv_g &=\int_{\Sn} \left(\sum_{i=0}^{2k} q_i f_i\right)^2 dv_g \\ 
    &= \sum_{i=0}^{2k} q^2_i\int_{\Sn} f_i^2 dv_g \\ 
    &= \sum_{i=0}^{2k} q^2_i \\ 
    &= 1,
\end{align*}
by expanding both sides in \eqref{inégalitél2k+1}, we obtain

\[
\lambda_{2k+1}(\Sn,g) + \sum_{i\neq j}\mathcal{G}_q(e_i,e_j) \leq \sum_{i=1}^{n+1} \int_{\Sn} \left|\nabla \left( X_{e_i}\circ \phi_{\xi_q} \sum_{i=0}^{2k} q_i f_i \right)\right|^2 dv_g.
\]
However, by construction when \(i \neq j\) we have
\[
\mathcal{G}_q(e_i,e_j) = 0.
\]
Therefore, inequality \eqref{inégalitél2k+1} becomes
\begin{align}
    \lambda_{2k+1}(\S^n,g) &\leq  \sum_{i=1}^{n+1}\int_{\Sn} \left|\nabla \left( X_{e_i}\circ \phi_{\xi_q} \left(\sum_{i=0}^{2k} q_i f_i\right) \right) \right|^2 dv_g.
\end{align}
Our next goal is to estimate this right term. We set $$v_i = X_{e_i}\circ \phi_{\xi_q} \quad \text{and} \quad u = \sum_{i=0}^{2k} q_i f_i,$$
so that
\begin{align}\label{énergieàestimer}
    \lambda_{2k+1}(\S^n,g) &\leq  \sum_{i=1}^{n+1}\int_{\Sn} \left|\nabla \left( v_i u \right) \right|^2 dv_g.
\end{align}
Next proposition shows that we can split this right term for an easier control.

\begin{prop}\label{proposition_decomposition_pd}
      For any two smooth functions $u$ and  $v$ in $\mathcal{C}^\infty(\S^n)$, we have  $$\int_{\Sn}|\nabla (v u) |^2dv_g= \int_{\Sn}v_i^2 u \Delta_g u \, dv_g + \int_{\Sn} u^2 |\nabla v |^2dv_g.$$
\end{prop}

\begin{proof}
The proof is just a successive integrations by parts: 
\begin{align*}
    \int_{\Sn}|\nabla (v u) |^2dv_g &= \int_{\Sn} v u \Delta_g(v u)dv_g \\
        &=\int_{\Sn} (v_i^2 u\Delta_g u + u^2 v \Delta_g v_i -2 v_i u \nabla u \cdot \nabla v)dv_g\\
        &=\int_{\Sn} (v^2 u\Delta_g u + u^2 v \Delta_g v_i - \frac{1}{2} \nabla u^2 \cdot \nabla v^2)dv_g\\
        &=\int_{\Sn} (v^2 u\Delta_g u + u^2 v \Delta_g v_i - \frac{1}{2} u^2\Delta_g v^2)dv_g\\
        &=\int_{\Sn} (v^2 u\Delta_g u + u^2 v \Delta_g v - \frac{1}{2} u^2( 2 v\Delta_g v -2 |\nabla v|^2))dv_g\\
        &=\int_{\Sn} (v^2 u\Delta_g u + u^2|\nabla v|^2)dv_g.
\end{align*}
\end{proof}
Applying Proposition \ref{proposition_decomposition_pd} to the two functions $u$ and $v_i$ involved in inequality \eqref{énergieàestimer}, and observing that $\sum_{i=1}^{n+1} v_i^2 = 1$, we get
\begin{equation}\label{terms_to_estimate}
       \lambda_{2k+1}(\Sn,g)\leq  \int_{\Sn} u \Delta_g u \, dv_g +\sum_{i=1}^{n+1}\int_{\Sn} u^2 |\nabla v_i |^2 dv_g.
\end{equation}
We now estimates each of the terms appearing in the right hand of this last inequality. To estimate $\int_{\Sn} u \Delta_g u \, dv_g$, we proceed as follows
\begin{align*}
\int_{\Sn} u \Delta_g u \, dv_g &= \int_{\Sn} \left(\sum_{i=0}^{2k} q_i f_i \right) \Delta_g\left(\sum_{i=0}^{2k} q_i f_i\right) dv_g \\ 
&= \int_{\Sn} \left(\sum_{i=0}^{2k} q_i f_i \right) \left(\sum_{i=0}^{2k} q_i \lambda_{i}(\Sn,g) f_i\right) dv_g \\ 
&=  \sum_{i=0}^{2k} q^2_i \lambda_{i}(\Sn,g) \int_{\Sn}  f^2_i \, dv_g \\
&\leq \lambda_{2k}(\Sn,g) \sum_{i=0}^{2k} q^2_i\int_{\Sn}  f^2_i \, dv_g \\
&= \lambda_{2k}(\Sn,g).
\end{align*} 
Substituting into \eqref{terms_to_estimate}, we deduce
  
\begin{equation}\label{l_k-l_k-1}
      \lambda_{2k+1}(\Sn,g)- \lambda_{2k}(\Sn,g) \leq \sum_{i=1}^{n+1}\int_{\Sn} u^2 |\nabla v_i |^2 dv_g.
\end{equation}
It remains to control
$ \sum_{i=1}^{n+1}\int_{\Sn} u^2 |\nabla v_i|^2 \, dv_g.$ To achieve this we apply Hölder’s inequality,

\begin{align*}
\sum_{i=1}^{n+1}\int_{\Sn} u^2 |\nabla v_i|^2 \, dv_g &\leq \left( \int_{\Sn} \left( \sum_{i=1}^{n+1} |\nabla_g v_i|^2 \right)^\frac{n}{2} dv_g \right)^\frac{2}{n} 
 \left( \int_{\Sn} |u|^\frac{2n}{n-2} \, dv_g \right)^\frac{n-2}{n} \\
\end{align*}

Since $g \in [g_0]$, by conformal invariance we have
\begin{align*}
    \left( \int_{\Sn} \left( \sum_{i=1}^{n+1} |\nabla_g v_i|^2 \right)^\frac{n}{2} dv_g \right)^\frac{2}{n} &=\left( \int_{\Sn} \left( \sum_{i=1}^{n+1} |\nabla_{g} X_{e_i} \circ \phi_\xi|^2 \right)^\frac{n}{2} dv_{g} \right)^\frac{2}{n} \\ &= n w_n^\frac{2}{n}.
\end{align*}

Hence, we obtain
\[
\sum_{i=1}^{n+1} \int_{\Sn} u^2 |\nabla v_i|^2 \, dv_g \leq n w_n^\frac{2}{n} \|u\|^2_{\frac{2n}{n-2}}.
\]

Finally, returning to inequality \eqref{l_k-l_k-1}, we deduce

\begin{equation}\label{inégsanscontinuité}
    \lambda_{2k+1}(\Sn,g) - \lambda_{2k}(\Sn,g) \leq n w_n^\frac{2}{n} \|u\|^2_{\frac{2n}{n-2}}.
\end{equation}

Once we reach this step, the proof of Theorem \ref{Thm1} follows from the estimate we know for \(\|u\|^2_{\frac{2n}{n-2}}\), which is a Sobolev norm with critical exponent.

\subsection{Proof of Theorem \ref{Thm1}}

\begin{proof}[Proof of Theorem \ref{Thm1}]
From inequality \eqref{inégsanscontinuité} we've just obtained, we have
\[
\lambda_{2k+1}(\Sn, g) - \lambda_{2k}(\Sn, g) \leq n w_n^\frac{2}{n} \| u \|^2_{\frac{2n}{n-2}}.
\]
We apply Hebey's estimate \eqref{Hebey-Aubinsphere} on the critical Sobolev norm to obtain
\[
\lambda_{2k+1}(\Sn, g) - \lambda_{2k}(\Sn, g) \leq n w_n^\frac{2}{n} \left( K(n, 2)^2 \| \nabla u \|_2^2 + \frac{\max_{\Sn} S_g}{n(n-1) w_n^\frac{2}{n}} \| u \|_2^2 \right).
\]
The proof is completed by observing that
\[
\| \nabla u \|_2^2 \leq \lambda_{2k}(\S^n,g), \quad \text{and} \quad \| u \|_2^2 = 1,
\]
which allows us to deduce that
\[
\lambda_{2k+1}(\Sn, g) - \lambda_{2k}(\Sn, g) \leq n w_n^\frac{2}{n} \left( K(n, 2)^2 \lambda_{2k}(\Sn, g) + \frac{\max_{\Sn} S_g}{n(n-1) w_n^\frac{2}{n}} \right).
\]
Theorem \ref{Thm1} follows after simplifications.

\end{proof}

\section{The case of an arbitrary closed Riemannian manifold}\label{sectionfinale}

Let \((M, g)\) be a closed manifold of dimension \(n \geq 3\), and let \(\phi: (M,g) \to \mathbb{S}^m\) be a conformal immersion. Our objective in this section is to ensure that the construction developed for the sphere extends naturally to \(M\) via \(\phi\). This approach, which played a central role in our previous work \cite{EddaoudiGirouard}, introduces a crucial additional geometric quantity: the $m$-conformal volume $V_c(m,M,C)$.

Set \( \{f_i\} \) to be an orthonormal basis of eigenfunctions for $L^2(M,g)$. More precisely, we aim  to establish that the corresponding inequality to \eqref{inégsanscontinuité} on $(M,g)$ takes the form
\begin{equation}\label{inegalivoconfme}
    \lambda_{2k+1}(M,g) - \lambda_{2k}(M,g) \leq n V_c(m,M,C)^{\frac{2}{n}} \| u\|^2_{\frac{2n}{n-2}},
\end{equation}
where \(u = \sum_{i=0}^{2k} q_i f_i\) and \(q \in \mathbb{S}^{2k}\).

To this end, let \( p \in \mathbb{S}^{2k} \) and define $\m_p$ as the analogous density on \( (M, g) \)  to \eqref{densitédmu},
\[
\m_p := \left( \sum_{i=0}^{2k} p_i f_i \right)^2 dv_g.
\]
By considering the pushforward measure \( \phi_* \m_p \) on the sphere \( \mathbb{S}^m \), we can apply the same standard topological argument to deduce the existence of a point \( \xi_p \in \mathbb{B}^{m+1} \) that satifies
\[
\int_{M} \phi_{\xi_p} \circ \phi \, \m_p = 0,
\]
where \( \phi_{\xi_p} \) is the conformal automorphism of $\S^m$ defined in \eqref{defiphi_xi}.

Next we set \(\{b_i\}\) to be an orthonormal basis of \(\mathbb{R}^{m+1}\), and we consider a map  \(\mathcal{\FF}: \mathbb{S}^{2k} \to \mathbb{R}^{2k+1}\) as
\begin{equation}
    \mathcal{\FF}(p) =
\begin{bmatrix}  
\int_{M} \h(p,x) f_0(x) \, dv_g(x) \\  
\vdots \\  
\int_{M} \h(p,x) f_{2k}(x) \, dv_g(x)  
\end{bmatrix},
\end{equation}
where  the function \(\h: \mathbb{S}^{2k} \times M \to \mathbb{R}\) is defined as
\[
\h(p,x) := \left(\sum_{i=1}^{m+1} X_{b_i} \circ \phi_{\xi_p}(x) \circ \phi \right) \left(\sum_{i=0}^{2k} p_i f_i(x)\right).
\] 
By adapting the idea behind Lemma \ref{applicationF} to this setting, we can prove similarly that there exists a point $q \in \S^{2k}$ such that 

$$
\mathcal{\FF}(q)=0.
$$
Therefore, the function $\h(q,x)$ is admissible in the variational charaterization of $\l_{2k+1}(M,g)$ \eqref{caravariation},
\[
\lambda_{2k+1}(M,g) \int_{M} \h(q,x)^2 \, dv_g \leq \int_{M} |\nabla \h(q,x)|^2 \, dv_g.
\]
Next we choose an orthonormal basis $\{e_i\}$ that diagonalizes the quadratic form 

{\scriptsize
\begin{align*}
\mathcal{\tilde{G}}_q(v,w) &= \lambda_{2k+1}(M,g) \int_{M} X_v \circ \phi_{\xi_q} \circ \phi  X_w \circ \phi_{\xi_q} \circ \phi \, \m_q \\
&\quad - \int_{M} \nabla \Big(X_v \circ \phi_{\xi_q} \circ \phi \sum_{i=0}^{2k} q_i f_i \Big) 
\cdot \nabla \Big(X_w \circ \phi_{\xi_q} \circ \phi  \sum_{i=0}^{2k} q_i f_i \Big) \, dv_g,
\end{align*}
}
and we follow the same line of argument as in the previous section to obtain a new function
$$\tilde{f}=  \sum_{i=1}^{m+1} X_{e_i} \circ \phi_{\xi_p} \circ \phi  \sum_{i=0}^{2k} q_i f_i$$ such that it is also an admissible function for $\l_{2k+1}(M,g)$,
\[
\lambda_{2k+1}(M,g) \int_{M} \tilde{f}^2 \, dv_g \leq \int_{M} |\nabla \tilde{f}|^2 \, dv_g.
\]
The outline of the previous section dealing with the Rayleigh quotient of $\tilde{g}$ also remain valid until we reach the following inequality, which corresponds to \eqref{inégsanscontinuité} in the sphere case, 
\begin{equation}\label{inégvoluconforme}
\lambda_{2k+1}(M,g) - \lambda_{2k}(M,g) \leq \Bigg(\int_{M} \Big( \sum_{i=1}^{m+1} |\nabla_{g} X_{e_i} \circ \phi_{\xi_q} \circ \phi|^2 \Big)^{\frac{n}{2}} dv_{g} \Bigg)^{\frac{2}{n}} \|u\|^2_{\frac{2n}{n-2}},
\end{equation}  
where \( u = \sum_{i=0}^{2k} q_i f_i \).  

Here we proceed by observing that in the setting of a closed Riemannian manifold, the expression of the integral
\[
\left( \int_{M} \left( \sum_{i=1}^{m+1} \left| \nabla_{g} \left( X_{e_i} \circ \phi_{\xi_q} \circ \phi \right) \right|^2 \right)^{\frac{n}{2}} dv_{g} \right)^{\frac{2}{n}}
\]
can be rewritten in terms of the pullback metric \( (\phi_{\xi_q} \circ \phi)^* g_0 \), namely
\begin{align*}
\Bigg(\int_{M} \Big( \sum_{i=1}^{m+1} |\nabla_{g} X_{e_i} \circ \phi_{\xi_q} \circ \phi|^2 \Big)^{\frac{n}{2}} dv_{g} \Bigg)^{\frac{2}{n}}  
&= n \Big(\text{vol}((\phi_{\xi_q} \circ \phi)^* g_0)\Big)^{\frac{2}{n}} \\
&\leq n \Big(V_c(m,\phi)\Big)^{\frac{2}{n}}.
\end{align*}  
We then take the infimum over all conformal immersions \(\phi : M \to \mathbb{S}^m\) in \eqref{inégvoluconforme} to deduce inequality \eqref{inegalivoconfme} 
\[
\lambda_{2k+1}(M,g) - \lambda_{2k}(M,g) \leq n V_c(m,M,C)^{\frac{2}{n}} \| u\|^2_{\frac{2n}{n-2}}.
\]  
Finally, we are now able to complete the proof of Theorem \ref{Thm2}, Theorem \ref{Thm3}, and Theorem \ref{Thm1bis} using the appropriate estimates on the critical Sobolev norm $\| u\|_{\frac{2n}{n-2}}$. 
\begin{proof}[Proof of Theorem \ref{Thm2}]
Let \((M,g)\) be a closed Riemannian manifold of dimension \(n \geq 3\), and assume that \(Ric_g \geq (n-1)a^2\) for some \(a > 0\).  

Applying Ilias' estimates from \eqref{IliasRicc} to inequality \eqref{inegalivoconfme}, we obtain
\[
\lambda_{2k+1}(M,g) - \lambda_{2k}(M,g) \leq  \frac{4 n V_c(m,M,C)^{\frac{2}{n}}}{n(n-2)a^2 \text{vol}(M,g)^{\frac{2}{n}}} \|\nabla u\|^2_2 
+  \frac{n V_c(m,M,C)^{\frac{2}{n}} }{\text{vol}(M,g)^{\frac{2}{n}}} \|u\|^2_2.
\]
Since
\[
\int_M |\nabla u|^2 \, dv_g  
= \sum_{i=0}^{2k} q_i^2 \int_M |\nabla f_i|^2 \, dv_g  
\leq \sum_{i=0}^{2k} q_i^2 \lambda_i(M,g)  
\leq \lambda_{2k}(M,g),
\]
and
\[
\|u\|^2_{2} = \int_M \bigg| \sum_{i=0}^{2k} q_i f_i \bigg|^2 dv_g  
= \sum_{i=0}^{2k} q_i^2 \int_M f_i^2 dv_g  
= 1.
\]
After regrouping the eigenvalue terms, we deduce
\[
\li- \la \left( 1 + \frac{4 V_c(m,M,C)^\frac{2}{n}}{(n-2) a^2 \text{vol}(M,g)^\frac{2}{n}} \right) \leq n V_c(m,M,C)^\frac{2}{n}.
\]
which completes the proof.
\end{proof}

\begin{proof}[Proofs of Theorem \ref{Thm3} and Theorem \ref{Thm1bis}]
Let \((M,g)\) be a closed Riemannian manifold of dimension \(n \geq 3\).

We apply Ilias' estimates from  \eqref{Iliasgene} to inequality \eqref{inégvoluconforme}, and the result follows in the same manner
\[
\lambda_{2k+1}(M,g) - \left( 1 + \frac{8 C^{*2} V_c(m,M,C)^{\frac{2}{n}}}{(n-2) C^2(M,g) w_n^{\frac{2}{n}}} \right) \lambda_{2k}(M,g) 
\leq 4 n V_c(m,M,C)^{\frac{2}{n}}.
\]
When the Yamabe constant $Y(M,[g])$ is sctriclty positive, we simply apply again the corresponding Sobolev inequality \eqref{SobolevY}, and Theorem \ref{Thm1bis} follows immediately.
\end{proof}

\section*{Acknowledgments}
The author would like to thank his PhD advisor, Alexandre Girouard, for his valuable suggestions, which have greatly improved the clarity of this work. The author also thanks Bruno Colbois for his insightful discussions.
\bibliographystyle{plain}
\bibliography{biblio}
\bigskip
\end{document}